\documentclass{amsart}

\usepackage{amssymb}
\usepackage{amsthm}
\usepackage{amsmath}
\usepackage[hidelinks]{hyperref}

\theoremstyle{plain}
 \newtheorem{theorem}{Theorem}[section]
 \newtheorem{lemma}[theorem]{Lemma}
 \newtheorem{corollary}[theorem]{Corollary}
 \newtheorem{proposition}[theorem]{Proposition}
 \newtheorem{claim}[theorem]{Claim}
\theoremstyle{definition}
 \newtheorem{definition}[theorem]{Definition}
 
 \newtheorem{example}[theorem]{Example}
 \newtheorem{remark}[theorem]{Remark}
 
\numberwithin{equation}{section}

\renewcommand{\a}{\alpha}
\renewcommand{\b}{\beta}
\newcommand{\g}{\gamma}

\renewcommand{\k}{\kappa}
\renewcommand{\o}{\omega}
\renewcommand{\l}{\lambda}

\newcommand{\sst}{\subseteq}

\newcommand{\cu}{\mathcal U}

\newcommand{\I}{\mathcal I}

\newcommand{\seq}[1]{\left<#1\right>}

\newcommand{\set}[1]{\left\{#1\right\}}
\newcommand{\abs}[1]{\left\vert#1\right\vert}

\newcommand{\id}{\operatorname{id}}

\renewcommand{\subset}{\subseteq}

\author[B. Kuzeljevi\'c]{Bori\v sa Kuzeljevi\'c}
\thanks{}
\address[Kuzeljevi\'c]{Department of Mathematics and Informatics, Faculty of Sciences, University of Novi Sad, Serbia.}
\email{borisha@dmi.uns.ac.rs}
\urladdr{\url{https://people.dmi.uns.ac.rs/\textasciitilde borisha}}
\dedicatory{}

\author[S. Milo\v sevi\'c]{Stepan Milo\v sevi\'c}
\thanks{}
\address[Milo\v sevi\'c]{Faculty of Technical Sciences, University of Novi Sad, Serbia.}
\email{stepanmilosevic@uns.ac.rs}

\keywords{Tukey reducibility, topological group, tightness of a space, basic order}
\subjclass[2010]{22A05, 54A20, 03E04}
\thanks{}

\title{Cofinal types and topological groups}

\date{\today}

\begin{document}

\begin{abstract}
    The purpose of this note is to start the systematic analysis of cofinal types of topological groups.
\end{abstract}

\maketitle

\section{Introduction}
In this note we start a systematic analysis of cofinal types of topological groups.
There has been much work on this topic, see for example \cite{feng, rejection, gartside, basicstevo, kakol}, however, what we are aiming in this paper is to provide a uniform language for these results, and to extend some of the known results.
For all undefined notions we refer the reader to Section \ref*{s:prel}.

When we say \emph{the cofinal type of a topological group}, we have in mind the cofinal type of a local base of the identity in a topological group.
Note that this makes sense as every two local bases of the identity of a topological group are cofinally similar. 
Let us also point out that since each topological group is a homogeneous topological space, in fact all local bases of all points in a topological group are cofinally similar.

Since Tukey maps are typically used to compare cofinalities of various directed sets, we will also be using these to compare cofinal types of topological groups.
Hence, we will say that a topological group $G$ is Tukey reducible to a topological group $G'$ if there is a Tukey map from a local base of the identity of $G$ to a local base of the identity of $G'$, and we will write $$G\le_T G'.$$
We will say that topological groups $G$ and $G'$ have the same cofinal type, or that they are cofinally similar, if $G\le_T G'$ and $G'\le_T G$.
Similarly, for a topological group $G$ and a directed set $D$, we will say that $G\le_T D$ if there is a Tukey map from a local base of the identity of $G$ to $D$.
All the other possible equalities and inequalities are defined in a similar manner.

Note that in this language, the well-known Birkhoff-Kakutani theorem states that a topological group $G$ is metrizable if and only if $G\le_T \o$.
Clearly, from a topological viewpoint, discrete groups (i.e. those groups $G$ with $G\equiv_T 1$) are the simplest possible, and one might be tempted to assume that metrizable non-discrete groups come as the next class in terms of topological complexity.
However, in our setting, since both $\o\not\le_T \o_1$ and $\o_1\not\le_T \o$ hold, we see that all topological groups whose cofinal type is $\o_1$ are incomparable with metrizable groups.
This shows that the scale for comparing the order of topological complexity of topological groups is not linear.  

\section{Preliminaries}\label{s:prel}

We use mostly standard terminology from set theory and topology.
For topological groups, our main reference is \cite{knjiga}.
Note, however, that there may be minor differences in the notation.

Let us point out that whenever we consider a topological group, its identity will be denoted $e$.
When there is a possibility for a confusion, the identity of a group $G$ will be denoted $e_G$.
In this paper, a topological group is a group $(G,\cdot)$ with Hausdorff topology $\tau$ such that both $\cdot$ and ${}^{-1}$ are continuous operations in $\tau$.
Typically, we will point out what is the operation in the group, and what is the topology on the group, only in cases when there is a danger of misunderstanding.

If $A$ is a set, then $\abs{A}$ denotes the cardinality of a set $A$.
For a cardinal $\k$ and a set $A$, we denote $[A]^{<\k}=\set{X\sst A: \abs{X}<\k}$ and $[A]^{\k}=\set{X\sst A:\abs{X}=\k}$.
Note that all other variations on this notion, for example $[A]^{\le\k}$, are defined analogously.
For a function $f:X\to Y$, and $A\sst X$ and $B\sst Y$ we define the direct image of $A$ via $f$ by $f[A]$, and the preimage of $B$ via $f$ by $f^{-1}[B]$.

Recall that $(P,\le_P)$ is \emph{a partial order} if $P$ is a set and $\le_P$ is a reflexive, antisymmetric and transitive binary relation on $P$.
If $(P,\le_P)$ is a partial order, we say that $X\sst P$ is \emph{unbounded} in $P$ if there is no $z\in P$ such that $x\le_P z$ for all $x\in X$, while we say that $Y\sst P$ is \emph{cofinal} in $P$ if for every $x\in P$ there is some $y\in Y$ such that $x\le_P y$.

\begin{definition}\label{usmeren_skup}
    We say that a partially ordered set $(D,\le_D)$ is \emph{a directed set} if for every $x$ and $y$ in $D$ there is some $z\in D$ such that $x\le_D z$ and $y\le_D z$.  
\end{definition}

We will always denote the ordering of a directed set $D$ by $\le_D$, so we will be able to write $D$ in place of $(D,\le_D)$.
Recall also that, for a regular infinite cardinal $\k$, a partially ordered set $D$ is $\k$-directed if for any $\a<\k$ and every collection $\set{x_{\xi}:\xi<\a}$ of $\a$ many elements from $D$, there is $x\in D$ such that $x_{\xi}\le x$ for each $\xi<\a$.
Thus, $D$ being directed means that $D$ is $\o$-directed in this language.

\begin{definition}
    Let $D$ and $E$ be directed sets. We say that $D$ is \emph{Tukey reducible} to $E$ if there is a map $f:D\to E$ such that for every unbounded set $X\sst D$ the set $f[X]$ is unbounded in $E$. 
    We write this $D\le_T E$, and we call such a map \emph{Tukey map}.
\end{definition}

Note that this is equivalent to saying that the preimage of every set bounded in $E$ is also bounded in $D$.
The following properties of directed sets and cofinal types, up to Remark \ref{countabledirected}, can be found in \cite{stevo}.

\begin{definition}
    Let $D$ and $E$ be directed sets.
    We say that a map $g:E\to D$ is \emph{cofinal} if for every $x\in D$ there is $y\in E$ such that $x\le_D g(z)$ whenever $y\le_E z$.
\end{definition}

\begin{lemma}
    Let $D$ and $E$ be directed sets.
    Then $D\le_T E$ if and only if there is a cofinal map $g:E\to D$.
\end{lemma}

\begin{lemma}\label{l:cofinalmap}
	Let $D$ and $E$ be directed sets. Then $f:D\to E$ is a cofinal map iff $f[A]$ is cofinal in $E$ for every $A$ which is cofinal in $D$.
\end{lemma}

\begin{definition}
    Let $D$ and $E$ be directed sets.
    If $D\le_T E$ and $E\le_T D$, then we say that $D$ and $E$ are \emph{cofinally similar} or that $D$ and $E$ have the same \emph{cofinal type}. We write this $D\equiv_T E$. 
\end{definition}

\begin{theorem}\label{cofinalsimilarity}
	Let $D$ and $E$ be directed sets. Then $D\equiv_T E$ iff there is a poset $P$ such that both $D$ and $E$ can be embedded as cofinal subsets of $P$.
\end{theorem}

\begin{remark}\label{cofinalsubset}
	Note that for directed sets $D$ and $E$, if $D$ is a cofinal subset of $E$, then $D\equiv_T E$.
\end{remark}

\begin{remark}\label{countabledirected}
	Note that if $D$ is a countable directed set, then $D\le_T \o$.
	In particular, either $D\equiv_T 1$ or $D\equiv_T \o$.
\end{remark}


\begin{definition}\label{d:kbox}
	Let $\set{X_i:i\in I}$ be a collection of topological spaces, and let $\k$ be a regular infinite cardinal.
	We define \emph{the $\k$-box topology} on $\prod_{i\in I}X_i$ as the topology given by the base
	$$\set{\bigcap_{i\in K}\pi_i^{-1}[U_i]: K\in [I]^{<\k}, (\forall i\in K)\ U_i\ \mbox{is open in }X_i}.$$
\end{definition}
   
Note that, with this definition, the product topology on $\prod_{i\in I}X_i$ is just $\o$-box topology, while the box topology is just $\abs{I}^+$-box topology.

\begin{theorem}\label{t:kbox}
	Let $\set{G_i:i\in I}$ be a collection of topological groups and $\k$ an infinite regular cardinal.
	For each $i\in I$, let $\mathcal{N}_i$ be a local base of the identity $e_i$ of the group $G_i$.
	Then the group $G=\prod_{i\in I}G_i$ with the $\k$-box topology is a topological group.
	The local base of the identity of this group is
	$$\mathcal{N}=\set{\bigcap_{i\in K}\pi_i^{-1}[B_i]:K\in [I]^{<\k},\ (\forall i\in K)\ B_i\in\mathcal{N}_i}.$$
\end{theorem}
   
\begin{proof}
   First we prove that $G$ with the $\k$-box topology is a topological group.
   Note that it is enough to show that the map $(x,y)\mapsto xy^{-1}$ is continuous.
   Let $U$ be a basic open set containing $ab^{-1}$.
   Let $a=\seq{a_i:i\in I}$ and $b=\seq{b_i:i\in I}$.
   Then $U=\bigcap_{i\in K}\pi_i^{-1}[U_i]$ where $\abs{K}<\k$ and each $U_i$ is open in $G_i$ ($i\in K$).
   Then $a_ib_i^{-1}\in U_i$ for $i\in K$, hence there are open sets $V_i$ and $W_i$ in $G_i$ such that $a_i\in V_i$ and $b_i\in W_i$ and $V_iW_i^{-1}\sst U_i$.
   Let $V=\bigcap_{i\in K}\pi_i^{-1}[V_i]$ and $W=\bigcap_{i\in K}\pi_i^{-1}[W_i]$.
   Clearly $a\in V$ and $b\in W$.
   Also, $VW^{-1}\sst U$, so $G$ is a topological group with this topology.

   Now we show that $\mathcal{N}$ is a local base of the identity.
   Clearly, each element of $\mathcal{N}$ is open in the $\k$-box topology.
   To prove that $\mathcal{N}$ is a local base of the identity for this topology, take any open set $U$ in $G$ containing $e$.
   Then there is a basic open set $W$ such that $e\in W\sst U$.
   Since $W$ is basic, there is a set $K\sst I$ of cardinality less than $\k$ such that $W=\bigcap_{i\in K}\pi_i^{-1}[U_i]$ where $U_i$ is an open set in $G_i$ for $i\in K$. 
   Since $e_i\in U_i$ for each $i\in K$, there are sets $B_i\in \mathcal{N}_i$ such that $e_i\in B_i\sst U_i$.
   Then $e\in \bigcap_{i\in K}\pi_i^{-1}[B_i]\sst W$ and $\bigcap_{i\in K}\pi_i^{-1}[B_i]\in\mathcal{N}$.
   This finishes the proof that $\mathcal{N}$ is a local base of $e$ in $G$.
\end{proof}

\section{Tukey order in the class of topological groups}

\begin{lemma}\label{well_defined}
    Let $G$ be a topological group with the identity $e$. 
    Let $(\mathcal{N}_1,\supseteq)$ and $(\mathcal{N}_2,\supseteq)$ be local bases of $e$ in $G$.
    Then $\mathcal{N}_1\equiv_T \mathcal{N}_2$.
\end{lemma}

\begin{proof}
	First, note that $\mathcal{N}_3 = \mathcal{N}_1 \cup \mathcal{N}_2$ is a local base of the identity $e$ of the topological group $G$. Hence a directed set under $\supseteq$.
	Clearly, both $\mathcal{N}_1$ and $\mathcal{N}_2$ are cofinal subsets of $\mathcal{N}_3$, so $\mathcal{N}_1\equiv_T \mathcal{N}_2$ by Theorem \ref{cofinalsimilarity}.
\end{proof}

\begin{definition}
    Let $G$ be a topological group with the identity $e$ and $D$ a directed set.
    We say that $G\le_T D$ ($G\ge_T D$ or $G\equiv_T D$) if there is a local base of $e$ in $G$, say $(\mathcal{N},\supseteq)$, such that $\mathcal{N}\le_T D$ ($\mathcal{N}\ge_T D$ or $\mathcal{N}\equiv_T D$).
\end{definition}

\begin{definition}
    Let $G$ be a topological group with the identity $e_G$ and $H$ a topological group with the identity $e_H$.
    We say that $G\le_T H$ ($G\ge_T H$ or $G\equiv_T H$) if there is a local base of $e_G$ in $G$, say $(\mathcal{N}_G,\supseteq)$, and a local base of $e_H$ in $H$, say $(\mathcal{N}_H,\supseteq)$ such that $\mathcal{N}_G\le_T \mathcal{N}_H$ ($\mathcal{N}_G\ge_T \mathcal{N}_H$ or $\mathcal{N}_G\equiv_T \mathcal{N}_H$).
\end{definition}

Note that these notions are well defined by Lemma \ref{well_defined}.

\begin{lemma}
	\label{H is Tukey quotient of G}
    Let $G$ be a topological group and $H$ a subgroup of $G$.
    Then $H\le_T G$.
\end{lemma}

\begin{proof}
	Since $H \le G$, then $e_{H} = e_{G}$.
	Denote the identity of both $H$ and $G$ with $e$.
	Let $\mathcal{N}_G$ be a local base of $e$ in $G$.
	Let us define $\mathcal{N}_H=\set{B\cap H: B\in \mathcal{N}_G}$.
	Note that $\mathcal{N}_H$ is a local base of $e$ in $H$.
	Define a map $f:\mathcal{N}_G\to \mathcal{N}_H$ as follows: for $B\in \mathcal{N}_G$ let $f(B)=B\cap H$.
	We will prove that $f$ is a cofinal map.
	Let $U\in \mathcal{N}_H$.
	Since it is open in $H$, it has the form $U=H\cap V$ for an open set $V\sst G$.
	Since $U\sst V$ we have $e\in V$.
	Hence, there is a $B\in \mathcal{N}_G$ such that $B\sst V$.
	Then $f(B)=B\cap H\sst V\cap H=U$, and for all $B'\in \mathcal{N}_G$ such that $B'\sst B$ we have $f(B')=B'\cap H\sst B\cap H\sst U$.
	Thus $f$ is a cofinal map, so $H\le_T G$.
\end{proof}

\begin{lemma}
    Let $G$ be a topological group and $H\le G$ an open subgroup of $G$.
    Then $G\equiv_T H$.
\end{lemma}
	
\begin{proof}
    Let $e$ be the identity of groups $G$ and $H$.
    Let $\mathcal{N}_{G}$ be a local base of $e$ in $G$.
    Then, since $H$ is open, $\mathcal{N}_{H} = \{B \cap H : B \in \mathcal{N}_{G}\}$ is a local base of $e$ both in $H$ and in $G$.
    Thus $H\equiv_T G$.
\end{proof}

\begin{lemma}\label{l:tukeyhom}
    Let $G$ and $H$ be topological groups, and let $\varphi:G\to H$ be an open and continuous homomorphism.
    Then $H\le_T G$.  
\end{lemma}

\begin{proof}
	 Let $e_{G}$ be the identity of $G$ and $e_{H}$ the identity of $H$.
	 Mapping $\varphi$ is a homomorphism, so $\varphi(e_{G}) = e_{H}$.
	 Let $\mathcal{N}_{G}$ be a local base at $e_{G}$. 
	 First, we claim that $\mathcal{N}_{H} = \{\varphi[U] : U \in \mathcal{N}_{G}\}$ is a local base at $e_{H}$. 
	 Clearly $\varphi[U]$ is an open neighborhood of $e_{H}$ since $\varphi$ is open. 
	 Let us fix an open neighborhood $V$ of $e_{H}$.
	 Since $\varphi$ is continuous at $e_G$, there is an open $U\in \mathcal{N}_G$ such that $\varphi[U]\sst V$.
	 Thus, $\mathcal{N}_H$ is a local base of $e_H$ in $H$.
	 
	 Next, we prove that $H \le_T G$. 
	 Define a mapping $g:\mathcal{N}_{G}\to \mathcal{N}_{H}$ by $g(U) = \varphi[U]$.
	 Let $\mathcal{K}$ be a cofinal subset of $\mathcal{N}_{G}$ and let $V$ be an open neighborhood in $\mathcal{N}_H$.
	 Then $V=\varphi[U]$ for some $U\in \mathcal{N}_G$.
	 Since $\mathcal K$ is cofinal in $\mathcal{N}_G$, there is a $W\in\mathcal K$ such that $W\sst U$.
	 Then $g(W)=\varphi[W]\in \mathcal{N}_H$ and $g(W)=\varphi[W]\sst V$.
	 Hence $g[\mathcal K]$ is a cofinal subset of $\mathcal{N}_H$, which means, by Lemma \ref{l:cofinalmap} that $g$ is a cofinal map and so $H\le_T G$.
\end{proof}

\begin{example}
	Note that the conclusion of Lemma \ref{l:tukeyhom} does not hold if one considers homomorphism which are only continuous and not open.
	In other words, continuous homomorphisms do not preserve cofinal types of topological groups.
	To see this, consider the symmetric group on $\o_1$, let us denote it by $S_{\o_1}$. Let $\mathcal O_T$ be the product topology on $S_{\o_1}$ and let $\mathcal O_{\o_1}$ be the $\o_1$-box topology (see Definition \ref{d:kbox} and the theorem just after it) on $S_{\o_1}$.
	Then $(S_{\o_1},\mathcal O_T)\equiv_T [\o_1]^{<\o}$ and $(S_{\o_1},\mathcal O_{\o_1})\equiv_T \o_1$, and clearly $\o_1<_T [\o_1]^{<\o}$.
	On the other hand $\id_{S_{\o_1}}:(S_{\o_1},\mathcal O_{\o_1})\to (S_{\o_1},\mathcal O_T)$ is a continuous homomorphism.
\end{example}

\begin{lemma}
    Let $G$ be a topological group and $N\lhd G$ a closed normal subgroup of $G$.
    Then $G/N\le_T G$.
\end{lemma}

\begin{proof}
    If $N$ is a normal closed subgroup of $G$ then $G/N$ is a well-defined topological group (by Theorem 1.5.3 in \cite{knjiga}).
	Quotient mapping (which is natural homomorphism) $\varphi:G\to G/N$ is an open and continuous homomorphism (again by Theorem 1.5.3 in \cite{knjiga}).
	By the previous Lemma, $G/N \le_T G$.
\end{proof}

\section{Tightness and similar properties}

Now we turn to three classes of spaces which are well-known, but we include definitions to make the note self contained as possible.

\begin{definition}\label{Frechet space}
	A topological space $X$ is called \emph{Fr\'{e}chet space} if for every $A \subset X$ and $x \in \overline{A}$ there is a sequence $\set{x_n:n<\o}\sst A$ converging to $x$.
\end{definition}

Note that if $G$ is a topological group, then it is Frech\'{e}t if for every $A\sst G$ with $e\in\overline{A}$ there is a sequence $\set{x_n:n<\o}\sst A$ such that $\lim_{n<\o}x_n=e$.

\begin{definition}\label{Sequential space}
	A topological space $X$ is called \emph{sequential space} if for every $A\sst X$ which is not closed there is a sequence $\set{x_n:n<\o}\sst A$ and a point $x\notin A$ such that $\lim_{n<\o}x_n=x$.
\end{definition}

\begin{definition}
	A topological space $X$ is \emph{countably tight} if for every $A\sst X$ and every $x\in \overline{A}$ there is a countable set $C\sst A$ such that $x\in\overline{C}$. 
\end{definition}

Note that if $G$ is a topological group, than it is countably tight if for every $A\sst G$ such that $e\in\overline{A}$, there is a countable set $C\sst A$ such that $e\in \overline{C}$. Let us also recall the well known implications for a topological space $X$:
\begin{center}
	$X$ is first-countable $\Rightarrow$ $X$ is Frech\'{e}t $\Rightarrow$ $X$ is sequential $\Rightarrow$ $X$ is countably tight.
\end{center}

Since in this paper we are interested in cofinal types of topological groups, we would like to point out here that neither of mentioned three properties are preserved under Tukey reducibility, as the following example shows. 

\begin{example}
	There are topological groups $G$ and $H$ such that $G\equiv_T H$, $G$ is Frech\'{e}t, and $H$ is not countably tight.
	Let $G$ be the group $$G=\set{x\in 2^{\o_1}:\abs{\set{\a<\o_1:x(\a)\neq 0}}<\o}$$ with the product topology (as in Example 5.10 in \cite{feng}).
	Note that the operation in this group is coordinatewise addition modulo $2$.
	This group is Frech\'{e}t and $G\equiv_T [\o_1]^{<\o}$.
	Let $H$ be the group $H=2^{\o_1}$ with the product topology and the same operation of coordinatewise addition.
	Then, clearly $G\equiv_T [\o_1]^{<\o}\equiv_T H$.
	However, $H$ is not countably tight as shown in Example 1.6.13 on page 51 in \cite{knjiga}.
\end{example}

However, there are properties of topological groups that are preserved under the Tukey reducibility.
The most basic example is metrizability.
Since $G$ being metrizable is equivalent to $G\le_T \o$, and since $\le_T$ is a transitive relation, it follows that if $G$ is metrizable and $H\le_T G$ then $H$ is metrizable as well.
The next proposition gives another simple example.
Recall that for a regular uncountable cardinal $\k$, a topological group $G$ is \emph{P$_{\k}$-group} if the intersection of less than $\k$ many neighborhoods of the identity $e$ is again a neighborhood of $e$. In this notation, a P-group is simply a P$_{\o_1}$-group.

\begin{proposition}
	Let $\k$ be a regular uncountable cardinal, $G$ a P$_{\k}$-group and $H$ a topological group such that $H \leq_T G$. Then, $H$ is a P$_{\k}$-group.
\end{proposition}

\begin{proof}
	Let $\mathcal{N}_{G}$ be a local base of $e_{G}$ in $G$ and $\mathcal{N}_{H}$ a local base of $e_{H}$ in $H$. 
	Since $H\le_T G$, there is a Tukey map $\varphi:\mathcal{N}_{H} \to \mathcal{N}_{G}$.
	
	Let $\a<\k$ and take any set $\set{U_{\xi}:\xi<\a}$ of $\a$ many neighborhoods of $e_H$.
	Since $\mathcal N_H$ is a local base of $e_H$, for each $\xi<\a$ there is $V_{\xi}\in \mathcal N_H$ such that $e\in V_{\xi}\sst U_{\xi}$.
	Suppose now that the proposition fails, i.e. that there is no $U\in \mathcal N_H$ such that $U\sst U_{\xi}$ for all $\xi<\a$.
	This means that $\set{U_{\xi}:\xi<\a}$ is unbounded in $(\mathcal N_H,\supseteq)$.
	Since $\varphi$ is a Tukey map, the set $\set{\varphi(U_{\xi}):\xi<\a}\sst \mathcal N_G$ is unbounded in $\mathcal N_G$.
	However, this is a contradiction with the assumption that $G$ is a P$_{\k}$-group, since that assumption guarranties there is $W\in \mathcal N_G$ such that $W\sst \varphi(U_{\xi})$ for all $\xi<\a$.
\end{proof}

Note that the previous proposition is actually a translation of the fact that, for a regular infinite cardinal $\k$, if $E$ is a $\k$-directed set and $D\le_T E$, then $D$ is $\k$-directed as well.
Basically the same proof shows both facts.

Another useful notion when analyzing cofinal types of topological groups is the following.

\begin{definition}\label{d:ideal}
	For a topological group $G$ with the identity $e$, we define $$\mathcal I_G=\set{A\sst G:e\notin \overline A}.$$
\end{definition}

\begin{remark}
	In the notation of the previous definition, $\mathcal I_G$ is an ideal of subsets of $G$, hence it is directed by $\subseteq$ relation.
	It is also clear that if $\mathcal{N}_G$ is a local base of $e$ then $(\mathcal{N}_G,\supseteq)\equiv_T (\mathcal I_G,\sst)$.
\end{remark}

We now move to higher analogues of countable tightness.

\begin{definition}\label{tihtness}
	For a topological space $X$, the \emph{tightness} of $X$ is the minimal cardinal $\k\geq\o$ with the property that for every set $A\subset X$ and every point $x\in\overline{A}$, there is $C\subset A$ such that $|C|\leq\k$ and $x\in\overline{C}$.
	
	The tightness of a space $X$ is denoted by $t(X)$.
\end{definition}

Let us connect the tightness of $G$ with its metrizability.

\begin{theorem}\label{t:tightness}
	Let $\k,\l$ be regular infinite cardinals and $G$ a topological group with $t(G)=\k$ and such that $G\le_T \l\times\k^+$.
	Then $G\le_T\l$.
\end{theorem}

\begin{proof}
	Let $e$ be the identity of the topological group $G$.
	Recall that $G\equiv_T \mathcal I_G$, where $\mathcal I_G$ is introduced in Definition \ref{d:ideal}.

	Let $f: \mathcal I_G \to \lambda \times \k^+$ be a Tukey map.
	For each $(\a, \eta) \in \lambda \times \k^+$, let
	$$
	A_{\a}^{\eta} = \{a \in \I_G : f(a) < (\a, \eta)\}.
	$$
	Note that $A_{\a}^{\eta}$ is bounded subset of $\I_G$ because $f$ is Tukey and it is an inverse image under $f$ of the bounded set $\{(\b,\xi) \in \lambda \times \k^+: (\b,\xi) < (\a,\eta)\}$. 
	Note also that for fixed $\a$: if $\xi < \eta < \k^+$, then $A_{\a}^{\xi} \sst A_{\a}^{\eta}$.
	
	Now, for each $\a < \lambda$, let us define
	$$
	B_{\a} = \bigcup \left(\bigcup_{\eta < \k^+}A_{\a}^{\eta} \right).
	$$
	Clearly, each $B_{\a}$ is a subset of $G$.
	
	Suppose for a moment that there is $\a < \lambda$ such that $B_{\a} \not \in \I_G$.
	This means that $e \in \overline{B_{\a}}$.
	Since $t(G) = \k$, there is a set $X \sst B_{\a}, |X| \le \k$ and $e \in \overline{X}$.
	Let us enumerate $X = \{x_{\b} : \b < \k\}$.
	Since $B_{\a} = \bigcup \left(\bigcup_{\eta < \k^+}A_{\a}^{\eta} \right)$, for each $\b < \k$ there is $a_{\b} \in \bigcup_{\eta < \k^+}A_{\a}^{\eta}$ such that $x_{\b} \in a_{\b}$.
	Hence, for each $\b < \k$, there is some $\g_{\b} < \k^+$ such that $x_{\b} \in a_{\b} \in A_{\a}^{\g_{\b}}$.
	Let $\g = \sup\{\g_{\b}: \b < \k\}<\k^+$.
	Clearly, for each $\b < \k$, $A_{\a}^{\g_{\b}} \sst A_{\a}^{\g}$.
	Hence, $a_{\b} \in A_{\a}^{\g}$ for each $\b < \k$.
	Since $A_{\a}^{\g}$ is bounded, there is a set $B \in \I_G$ such that $a_{\b} \sst B$ for each $\b < \k$.
	In particular, we have that $X \sst B$.
	Since $e \in \overline{X}$, this implies that $e \in \overline{B}$, which contradicts the assumption that $B \in \I_G$.
	Thus, $B_{\a} \in \I_G$ for each $\a < \lambda$.
	
	Now we will prove that $\{B_{\a} : \a < \lambda\}$ is cofinal in $\I_G$.
	This will show that $\I_G \leq_{T} \lambda$ and will finish the proof.
	Let $S \in \I_G$. 
	Then $f(S) = (\a, \xi)$ for some $(\a,\xi) \in \lambda \times \k^+$.
	This means that 
	$$
	S \in A_{\a+1}^{\xi+1} \sst \bigcup_{\eta < \k^+}A_{\a+1}^{\eta}.
	$$
	Finally, this implies that $S \sst \bigcup\left(\bigcup_{\eta < \k^+}A_{\a+1}^{\eta} \right) = B_{\a+1}\in \mathcal I_G$, which proves that $\{B_{\a}: \a < \lambda\}$ is cofinal in $\I_G$.
\end{proof}

Note also the following immediate corollaries.

\begin{corollary}
	Let $\k$ be an infinite regular cardinal and $G$ a topological group with $t(G)=\k$ and such that $G\le_T \o\times\k^+$.
	Then $G$ is metrizable.
\end{corollary}

\begin{corollary}\label{cor:metrtight}
	If $G$ is a countably tight topological group and $G\le_T \o\times\o_1$, then $G$ is metrizable.
	In particular if a countably tight group $G$ is Tukey reducible to $\o_1$ then it is metrizable.
\end{corollary}

Recall the following theorem from \cite{stevo}.
The statement is adapted to the more uniform treatment of consequences of the Proper Forcing Axiom, as in \cite{dichotomies}.
Note that all the relevant notions are thoroughly explained there (in \cite{dichotomies}).

\begin{theorem}\label{t:stevo}
    The P-ideal dichotomy and $\mathfrak p>\o_1$ imply that every directed set of cofinality at most $\o_1$ is cofinally similar to one of the following: $1$, $\o$, $\o_1$, $\o\times \o_1$, or $[\o_1]^{<\o}$.
\end{theorem}

Hence, Theorem \ref{t:stevo} and Corollary \ref{cor:metrtight} jointly imply that under P-ideal dichotomy: if $\mathfrak{p}>\o_1$, then every countably tight topological group of weight $\o_1$ is either metrizable or $[\o_1]^{<\o}$ is Tukey reducible to $G$.
This should be compared with the result of Todorcevic from \cite{pid}, that under P-ideal dichotomy, every Frech\'{e}t ($\alpha_1$)-point $x$ in a topological space $X$ is first countable in $X$ unless it has $[\o_1]^{<\o}$ as Tukey quotient.

\section{Metrizability of countably tight groups}

We now provide additional conditions which ensure that a countably tight topological group is metrizable.
First we introduce a couple of notions.

\begin{definition}
	We say that a directed set $(D,\le)$ is \emph{strongly basically generated} if there is a metric $\rho$ on $D$ such that $(D,\rho)$ is a separable metric space and that for every sequence $\set{d_n:n<\o}\subseteq D$ converging to some $d\in D$, there is $d^*\in D$ such that $\rho(d,d^*)\le\sup\set{\rho(d,d_n):n<\o}$ and $d_n\le d^*$ for each $n<\o$.
\end{definition}

For a topological group $G$, we say that $G$ is \emph{strongly basically generated} if $G\le_T D$ for some strongly basically generated directed set $D$.
The next result, Theorem \ref{t:main} and its proof, should be compared to the result (and the proof) of Feng and Dow \cite[Theorem 3.8]{dowfeng} that if $X$ is a compact space with countable tightness and $K(M)$-base for some separable metric space $M$, then $X$ is first-countable.
Here $K(M)$ stands for the collection of all compact subsets of $M$ ordered by inclusion.
What is used in their proof is that $K(M)$ is strongly basically generated.
Before we prove Theorem \ref{t:main}, we give the core of the proof in a lemma about general topological spaces, which may be useful on its own.

\begin{theorem}\label{l:main}
	Suppose that $X$ is regular, locally countably compact, and countably tight topological space.
	Let $x\in X$ be such that its neighborhood filter $\mathcal F_x$ is Tukey reducible to some strongly basically generated directed set $D$.
	Then $x$ has a countable local base in $X$.
\end{theorem}

\begin{proof}
	Observe first that if $x$ is an isolated point in $X$, then $\set{\set{x}}$ is a local base at $x$.
	Thus we may assume that $x$ is not an isolated point in $X$.

	
	Let $f: \mathcal F_{x} \to D$ be a Tukey reduction.
	Define $g: D \to \mathcal F_{x}$ by
	$$g(d) = \bigcap \set{a \in \mathcal F_{x} : f(a) \leq_{D} d}.$$
	First let us explain why $g$ is a well defined function.
	Since the set $$d\!\downarrow = \set{x \in D : x \leq_{D} d}$$ is bounded and $f$ is a Tukey function, the inverse image of $d\!\downarrow$ under $f$ is also bounded, i.e. the set 
	$$
	f^{-1}[d\!\downarrow] = \set{a \in \mathcal F_{x} : f(a) \leq_{D} d}
	$$
	is bounded in $\mathcal F_x$.
	Hence there is $b \in \mathcal F_{x}$ such that $b \sst a$ for all $a \in \mathcal F_{x}$ satisfying $f(a) \leq_{D} d$.
	Thus $b \sst \bigcap\set{a \in \mathcal F_{x} : f(a) \leq_{D} d}$, which implies that $\bigcap\set{a \in \mathcal F_{x} : f(a) \leq_{D} d}\in\mathcal F_x$ and so the function $g$ is well defined.
	
	Now we will show that $g$ is monotone and has a cofinal range.
	First we show the monotonicity of $g$: Let $d, e \in D$ and let $d \leq_{D} e$. Then
	$$
	g(d) = \bigcap\set{a \in \mathcal F_{x} : f(a) \leq_{D} d} \supseteq \bigcap\set{a \in \mathcal F_{x} : f(a) \leq_{D} e} = g(e),
	$$
	i.e. $g$ is a monotone function.

	Next we show that the range of $g$ is cofinal: Let $b \in \mathcal F_{x}$.
	Then $f(b) \in D$ and 
	$$
	g(f(b)) = \bigcap\set{a \in \mathcal F_{x} : f(a) \leq_{D} f(b)}.
	$$
	Clearly $b \in \set{a \in \mathcal F_{x}: f(a) \leq_{D} f(b)}$, so
	$$
	g(f(b)) = \bigcap\set{a \in \mathcal F_{x} : f(a) \leq_{D} f(b)} \sst b.
	$$
	Thus, $g[D]$ is cofinal in $\mathcal F_x$.

	Recall that we consider $D$ as a separable metric space, so let $\rho$ be the corresponding metric which makes $D$ strongly basically generated.
	For $d \in D$ and $\epsilon > 0$, let
	$$
	B_{d}(\epsilon) = \set{e \in D : \rho(e,d) < \epsilon},
	$$
	i.e. $B_d(\epsilon)$ is an open ball centered at $d$ of radius $\epsilon$.
	Next, we prove that for each $d\in D$, there is $n<\o$ such that $g[B_d(1/n)]$ has nonempty interior.
	We will first need a claim.
	For $d\in D$, let us define
	$$
	S(d)=\bigcup_{n<\o}\left(\bigcap_{e\in B_d(1/n)}g(e)\right).
	$$
	Define also $S_n(d)=\bigcap g\left[B_d(1/n)\right]$ and note that $S(d)=\bigcup_{n<\o}S_n(d)$.

	\begin{claim}\label{c:laksi}
		For each $d\in D$, there is an open set $U$ such that $x\in U\sst S(d)$.
	\end{claim}
	\begin{proof}
		Suppose that the claim fails.
		Then $x\in\overline{X\setminus S(d)}$ for some $d\in D$.
		Since $X$ is countably tight, there is a set $\set{y_n:n<\o}\sst X\setminus S(d)$ such that $x\in\overline{\set{y_n:n<\o}}$.
		Note that $$X\setminus S(d)=\bigcap_{n<\o}\left(\bigcup_{e\in B_d(1/n)}X\setminus g(e)\right).$$
		Hence, for each $n$, there is some $d_n\in B_d(1/n)$ such that $y_n\notin g(d_n)$.
		Since the sequence $\set{d_n:n<\o}$ is converging to $d$ and $D$ is strongly basically generated, there is $d^*\in D$ such that $d_n\le d^*$ for each $n<\o$.
		Since $g$ is monotone, this means that $x\in g(d^*)\sst g(d_n)$ for each $n<\o$.
		Thus, we found a neighborhood of $x$ containing no point from the set $\set{y_n:n<\o}$.
		This contradicts the choice of the set $\set{y_n:n<\o}$. 
	\end{proof}
	Now we improve the conclusion of Claim \ref{c:laksi} to a smaller set, as promised.
	\begin{claim}\label{c:tezi}
		For each $d\in D$, there are open set $V$ and $n<\o$ so that $x\in V\sst S_n(d)$.
	\end{claim}
	\begin{proof}
		Fix $d\in D$.
		Let $U$ be as in the conclusion of Claim \ref{c:laksi}, i.e. $U$ is open such that $x\in U\sst S(d)$.
		Since $X$ is regular and locally countably compact, there is an open set $V$ such that $x\in V\sst \overline{V}\sst U$ and that $\overline{V}$ is countably compact.
		We will prove that $V$ is as required in the claim.
		Suppose it is not the case.
		Then for every $n<\o$, $V\not\subseteq S_n(d)$ so there is $y_n\in V\setminus S_n(d)$.
		Note that $y_n$'s can be chosen to be different, i.e. the set $A=\set{y_n:n<\o}$ is infinite.
		Since $A\sst V$, we have $\overline{A}\sst \overline{V}\sst S(d)$.
		
		One the other hand, since $\overline{V}$ is countably compact, there is at least one limit point $y$ of $A$.
		The point $y$, being a limit point of $A$, is a limit point of each $A_m=\set{y_n:n>m}$ (for $m<\o$).
		Consider now the set $A_m$ for some $m<\o$.
		We know that for each $n>m$:
		$$y_n\in X\setminus S_n(d)=X\setminus\bigcap_{e\in B_d(1/n)}g(e)=\bigcup_{e\in B_d(1/n)}X\setminus g(e).$$
		From this, it follows that for each $n>m$ there is some $d_n\in B_d(1/n)$ such that $y_n\in X\setminus g(d_n)$.
		Clearly the sequence $\set{d_n:n>m}$ is converging to $d$.
		Since $D$ is strongly basically generated, there is some $d_m^*$ such that $d_n\le d_m^*$ for each $n>m$, and that moreover $\rho(d_m^*,d)<\frac{1}{m}$.
		Note that by monotonicity of $g$, this means that $y_n\in X\setminus g(d_m^*)$ for each $n>m$, which further implies that $y\in X\setminus g(d_m^*)$. 
		Now from $d_m^*\in B_d(1/m)$ we have $g(d_m^*)\supseteq \bigcap_{e\in B_d(1/m)}g(e)=S_m(d)$, which gives
		$$y\in X\setminus g(d_m^*)\sst X\setminus S_m(d).$$
		To summarize, we proved that for every $m<\o$, the point $y$ belongs to $X\setminus S_m(d)$.
		Thus $$y\in \bigcap_{m<\o}X\setminus S_m(d)=X\setminus \bigcup_{m<\o}S_m(d)=X\setminus S(d)$$ directly contradicting the fact that $\overline A\sst S(d)$.
		Hence, for some $n<\o$, we have that $x\in V\sst S(d)$.
	\end{proof}
	To finish the proof of the theorem first for each $d\in D$, take $n(d)<\o$ given by Claim \ref{c:tezi}.
	Now for each $d\in D$ there is an open set $V_d$ such that $x\in V_d\sst \bigcap g[B_d(1/n(d))]$.
	Take an open cover $\set{B_d(1/n(d)):d\in D}$ of the metric space $D$.
	By Lindel\"{o}f's theorem, there is a countable set $\set{d_m:m<\o}\sst D$ such that $\cu=\set{B_{d_m}(1/n(d_m)):m<\o}$ is an open cover of $D$.
	To prove that $\set{V_{d_m}:m<\o}$ is a countable local base of $x$, take any open set $W$ containing $x$.
	By cofinality of $g$, there is some $d\in D$ such that $g(d)\sst W$.
	Since $\cu$ is an open cover of $D$, there is some $m<\o$ such that $d\in B_{d_m}(1/n(d_m))$.
	Then $$x\in V_{d_m}\sst \bigcap g[B_{d_m}(1/n(d_m))]\sst g(d)\sst W.$$
	This proves the theorem.
\end{proof}

\begin{theorem}\label{t:main}
	Every countably tight, locally countably compact, and strongly basically generated topological group is metrizable.
\end{theorem}

\begin{proof}
	The proof follows directly from the fact that every topological group is regular, Theorem \ref{l:main}, and Birkhoff-Kakutani theorem.
\end{proof}

Note that the proof of previous two theorems closely follows the proof of Todorcevic in \cite[Theorem 2.1 and Corollary 2.2]{basicstevo} where he shows that every basic Frech\'{e}t topological group is metrizable.
The notion of a basic group is connected to our notion of strongly basically generated.
Namely, we say that a topological group is \emph{basic} if $G\le_T E$ for some basic order $E$ (basic order in a sense of Solecki and Todorcevic \cite[p. 1881]{slawek}).

\section{Products}

A natural question when considering cofinal types of topological groups is how they behave with respect to products.
It is well known that if $\set{G_i:i\in I}$ is a collection of topological groups, then $G=\prod_{i\in I}G_i$ with the product topology is a topological group with the identity $e$.
As one might assume, in this case, the cofinal type of $G$ is the finite support product of types of $G_i$'s.
In this section we prove a bit more general fact than this.
First we need a definition.

\begin{definition}\label{d:ksup}
 Let $\set{D_i:i\in I}$ be a collection of partially ordered sets.
 Suppose that $D_i$ has a minimum $0_i$ for each $i\in I$.
 Let $\k$ be an infinite regular cardinal.
 We define the $\k$-support product of posets $P_i$ as follows:
 $$\prod_{i\in I}^{\k-\operatorname{supp}}D_i=\set{x\in\prod_{i\in I}D_i: \abs{\set{i\in I: x_i\neq 0_i}}<\k}.$$
\end{definition}

Note that, as in the case of topology, the finite support product of $D_i$'s is simply $\prod_{i\in I}^{\o-\operatorname{supp}}D_i$.
Now we are able to present the central result of this section.
 
\begin{theorem}
 Let $\k$ be an infinite regular cardinal.
 Let $\set{G_i:i\in I}$ be a collection of topological groups such that $G_i\equiv_T D_i$ where $D_i$ is a directed set with the minimum $0_i$, for each $i\in I$.
 Suppose that $G=\prod_{i\in I}G_i$ with the $\k$-box topology, and that $D=\prod_{i\in I}^{\k-\operatorname{supp}}D_i$.
 Then $G\equiv_T D$.
\end{theorem}
 
\begin{proof}
	Let $\mathcal{N}_i$ be a local base of the identity $e_i$ of the group $G_i$.
 	According to Theorem \ref{t:kbox}, one local base of the identity of $G$ is given by sets $\bigcap_{i\in K}\pi_i^{-1}[B_i]$ where $K\in [I]^{<\k}$ and $B_i\in \mathcal{N}_i$ for all $i\in K$.
	Let $i\in I$.
	Since $G_i\equiv_T D_i$, i.e. $(\mathcal{N}_i,\supseteq)\equiv_T (D_i,\le_i)$, by Theorem \ref{cofinalsimilarity} there is a poset $P_i$ and embeddings $f_i:(\mathcal{N}_i,\supseteq)\to P_i$ and $g_i:D_i\to P_i$ such that $f_i[\mathcal{N}_i]$ and $g_i[D_i]$ are cofinal subsets of $P_i$.
	Consider now $P=\prod_{i\in I}^{\k-\operatorname{supp}}P_i$.
	To prove the theorem, it is sufficient to find embeddings $f:\mathcal{N}\to P$ and $g:D\to P$ such that $f[\mathcal{N}]$ and $g[D]$ are cofinal subsets of $P$.
 	Let $f$ and $g$ be given by $f(a)=\seq{f_i(\pi_i(a)):i\in I}$ and $g(d)=\seq{g_i(\pi_i(d)): i\in I}$.
	We will now show that $f$ and $g$ satisfy announced conditions.
 	
 	\begin{claim}\label{cl:emb}
  		The map $f$ is an embedding and $f[\mathcal{N}]$ is a cofinal subset of $P$.
 	\end{claim}
 	
 	\begin{proof}
		It is clear that $f$ is a well defined function.
		First we show that $f$ is one-to-one.
		Let $a\neq b$ be different elements of $\mathcal{N}$.
		This means tht there is an $i\in I$ such that $a_i\neq b_i$.
		Then $\pi_i(a)=a_i\neq b_i=\pi_i(b)$.
		Since $f_i$ is an embedding, $f_i(\pi_i(a))\neq f_i(\pi_i(b))$.
		Hence $f(a)\neq f(b)$ so $f$ is one-to-one.
 		
		Next we show that $f$ is a monotone function, i.e. that $a\le b$ in $\mathcal{N}$ implies $f(a)\le f(b)$.
  		Let $a\le b$ in $\mathcal{N}$.
		This means that $a_i\supseteq b_i$ for all $i\in I$, hence $\pi_i(a)=a_i\supseteq b_i=\pi_i(b)$.
		Since $f_i$ is an embedding for each $i\in I$, we have that $f_i(\pi_i(a))\le f_i(\pi_i(b))$ for all $i\in I$.
  		Then $f(a)\le f(b)$, as we were supposed to show, so $f$ is an embedding.

		Finally, we prove that $f[\mathcal{N}]$ is a cofinal subset of $P$.
		Let $x\in P$.
		Then $J=\set{i\in I: x_i\neq 0_i}$ is of cardinality less than $\k$.
		Since $f_i[\mathcal{N}_i]$ is a cofinal subset of $P_i$ for all $i\in I$, for each $i\in J$ we can choose an element $B_i\in\mathcal{N}_i$ such that $x_i\le f_i(B_i)$.
		Let us now consider the element $B=\bigcap_{i\in J}\pi_i^{-1}[B_i]$.
  		Since $J$ is of cardinality less than $\k$, and $B_i\in \mathcal{N}_i$ for all $i\in J$, we know that $B\in \mathcal{N}$.
  		Now we show that $x\le f(B)$ which will finish the proof.
  		It is sufficient to prove that $x_i\le \pi_i(f(B))$ for all $i\in I$.
  		If $i\in I\setminus J$, then $x_i=0_i$ so in that case $x_i\le \pi_i(f(B))$.
		Suppose now that $i\in J$.
  		Then $\pi_i(f(B))=f_i(\pi_i(B))=f_i(B_i)\ge x_i$, as promised.
 	\end{proof}
 	
 	\begin{claim}
  		The map $g$ is an embedding and $g[D]$ is a cofinal subset of $P$.
 	\end{claim}
 	
 	\begin{proof}
		Almost the same proof as in Claim \ref{cl:emb}.
 	\end{proof}
 	This proves the theorem.
\end{proof}
 
Note that this theorem is basically just a translation of the fact that cofinal similarity is preserved under $\k$-support products of directed sets.
Essentially the same argument would give this fact.
That proof can also be adapted to show that cofinal similarity is preserved under reduced products modulo any ideal over the index set.
The last theorem also has an immediate corollary, the one we announced at the beginning of this section.
 
\begin{corollary}
 Let $\set{G_i:i\in I}$ be a collection of topological groups.
 Let $G_i\equiv_T D_i$ where $D_i$ is a directed set with the minimum $0_i$, for each $i\in I$.
 Then $\prod_{i\in I}G_i$ with the product topology is cofinally similar to the finite support product of $D_i$'s, i.e. to $\prod_{i\in I}^{\o-\operatorname{supp}}D_i$.
\end{corollary}

\begin{example}
	Consider for a moment $S_{\k}$, the symmetric group on a regular infinite cardinal $\k$, i.e. the group of all bijections from $\k$ onto $\k$.
	This group can be equipped with the $\l$-box topology for any regular infinite cardinal $\l$.
	In this case we have $S_{\k}\equiv_T [\k]^{<\l}$ simply by the definition of $\l$-box topology.
	Then we have, for example:
	$$\prod_{n<\o}S_{\o_n}\equiv_T \prod_{n<\o}^{\o-\operatorname{supp}}[\o_n]^{<\o},$$
	when on the entire product is given the usual product topology.
\end{example}

\section{Acknowledgements}

The authors would like to thank Stevo Todorcevic for many useful conversations.
The first author was partially supported by grants from the Science Fund of the Republic of Serbia (Grant No. 7750027 - SMART), and the Ministry of Science, Technological Development and Innovation of the Republic of Serbia (Grants No. 451-03-66/2024-03/200125 and 451-03-65/2024-03/200125). The second author has been partially supported by the Ministry of Science, Technological Development and Innovation (Contract No. 451-03-65/2024-03/200156) and the Faculty of Technical Sciences, University of Novi Sad through project “Scientific and Artistic Research Work of Researchers in Teaching and Associate Positions at the Faculty of Technical Sciences, University of Novi Sad” (No. 01-3394/1).

\providecommand{\bysame}{\leavevmode\hbox to3em{\hrulefill}\thinspace}
\providecommand{\MR}{\relax\ifhmode\unskip\space\fi MR }
\providecommand{\MRhref}[2]{%
  \href{http://www.ams.org/mathscinet-getitem?mr=#1}{#2}
}
\providecommand{\href}[2]{#2}

\end{document}